\documentclass[11pt]{article}
\usepackage{amsfonts}
\usepackage{amsmath}
\usepackage{amssymb}
\usepackage{amsthm}
\usepackage{amsmath}
\usepackage{verbatim}
\usepackage{bibspacing}

\newtheorem{thm}{Theorem}[section]
\newtheorem*{thm*}{Theorem}
\newtheorem{cor}[thm]{Corollary}

\newtheorem{prop}[thm]{Proposition}
\newtheorem*{prop*}{Proposition}

\newtheorem*{conj*}{Conjecture}

\newtheorem*{dfn*}{Definition}
\theoremstyle{definition}
\newtheorem{rem}[thm]{\textbf{Remark}}
\newtheorem*{rmk*}{Remark}
\newtheorem*{fact*}{Fact}

\theoremstyle{proof}

\newcommand{\vol}{\textrm{Vol}}
\newcommand{\norm}[1]{\left\Vert#1\right\Vert}

\newcommand{\abs}[1]{\left\vert#1\right\vert}
\newcommand{\set}[1]{\left\{#1\right\}}
\newcommand{\brac}[1]{\left(#1\right)}
\newcommand{\scalar}[1]{\left \langle #1 \right \rangle}

\newcommand{\Real}{\mathbb{R}}

\newcommand{\volrad}{{\rm volrad}}
\renewcommand{\det}{{\rm det}}
\newcommand{\Cov}{{\rm Cov}}
\newcommand{\diam}{{\rm diam}}

\begin{document}

\title{On the mean-width of isotropic convex bodies and their associated $L_p$-centroid bodies}

\author{Emanuel Milman\textsuperscript{1}}

\footnotetext[1]{Department of Mathematics, Technion - Israel
Institute of Technology, Haifa 32000, Israel. Supported by ISF (grant no. 900/10), BSF (grant no. 2010288), Marie-Curie Actions (grant no. PCIG10-GA-2011-304066) and the E. and J. Bishop Research Fund. Email: emilman@tx.technion.ac.il.}

\date{}
\maketitle

\begin{abstract}
For any origin-symmetric convex body $K$ in $\Real^n$ in isotropic position, we obtain the bound:
\[
M^*(K) \leq C \sqrt{n} \log(n)^2 L_K ~,
\]
where $M^*(K)$ denotes (half) the mean-width of $K$, $L_K$ is the isotropic constant of $K$, and $C>0$ is a universal constant. This improves the previous best-known estimate $M^*(K) \leq C n^{3/4} L_K$. Up to the power of the $\log(n)$ term and the $L_K$ one, the improved bound is best possible, and implies that the isotropic position is (up to the $L_K$ term) an almost $2$-regular $M$-position. The bound extends to any arbitrary position, depending on a certain weighted average of the eigenvalues of the covariance matrix. Furthermore, the bound applies to the mean-width of $L_p$-centroid bodies, extending a sharp upper bound of Paouris for $1 \leq p \leq \sqrt{n}$ to an almost-sharp bound for an arbitrary $p \geq \sqrt{n}$. The question of whether it is possible to remove the $L_K$ term from the new bound is essentially equivalent to the Slicing Problem, to within logarithmic factors in $n$.  
\end{abstract}

\section{Introduction}

\begin{comment}
Notation: 
\begin{itemize}
\item $I_p(\mu) = \brac{\int |x|^p d\mu(x)}^{1/p}$.
\item $Z_p(\mu)$ is defined by the support function $h_{Z_p(\mu)}(\theta) = \brac{\int \abs{\scalar{x,\theta}}^p d\mu(x)}^{1/p}$. 
\item $W_p(K) = \brac{\int_{S^{n-1}} h_K(\theta)^p d\sigma(\theta)}^{1/p}$ is the $p$-th moment of the support-function w.r.t. the uniform measure $\sigma$ on the sphere. $W = W_1$.  
\item $G_{n,k}$ is the Grassmann manifold of $k$-dimensional subspaces of $\Real^n$, $P_F$ is the orthogonal projection onto $F \in G_{n,k}$, and $\pi_F \mu$ is the measure-projection of $\mu$ via $P_F$. 
\item $\Cov(\eta) = \int x \otimes x d\eta(x) - \int x d\eta(x) \otimes \int x d\eta(x)$ is the Covariance matrix of the probability measure $\eta$.
\item $L_\eta := \det \; \Cov(\eta)^{1/2k} / \norm{\eta}_{L^\infty}^{1/k}$ is the isotropic constant of the probability measure $\eta$ defined on $\Real^k$. 
\end{itemize}
\end{comment}

Throughout this work we work in Euclidean space $(\Real^n,\scalar{\cdot,\cdot})$. A convex body $K$ in $\Real^n$ is a compact convex set with non-empty interior, and the uniform probability measure on $K$ is denoted by $\lambda_K$. More generally, it is very useful to consider the larger class of log-concave probability measures $\mu$ on $\Real^n$, consisting of absolutely continuous probability measures having density $f_\mu$ of the form $\exp(-V)$ with $V : \Real^n \rightarrow \Real \cup \set{+\infty}$ convex. We denote by $\Cov(\mu)$ the covariance matrix of $\mu$, given by $\Cov(\mu) := \int x \otimes x \; d\mu(x) - \int x \; d\mu(x) \otimes \int x \; d\mu(x)$. We will say that $\mu$ is isotropic if its barycenter is at the origin and $\Cov(\mu)$ is the identity matrix $Id$. We will say that a convex body $K$ is isotropic if $K$ has volume one and $\lambda_{K / L_K}$ is isotropic for an appropriate constant $L_K > 0$, i.e. if its barycenter is at the origin and $\Cov(\lambda_K) = L_K^2 Id$. It is easy to see that by applying an affine transformation, any convex body may be brought to isotropic ``position", which is unique up to orthogonal transformations \cite{Milman-Pajor-LK}; the isotropic constant $L_K$ is thus an affine invariant associated to any convex body $K$. See Bourgain \cite{BourgainMaximalFunctionsOnConvexBodies,Bourgain-LK} and Milman--Pajor \cite{Milman-Pajor-LK} 
for background on the yet unresolved Slicing Problem, which is concerned with obtaining a dimension independent upper-bound on $L_K$. The current best-known estimate $L_K \leq C n^{1/4}$ is due to B. Klartag \cite{KlartagPerturbationsWithBoundedLK}, who improved the previous estimate $L_K \leq C n^{1/4} \log(n)$ of J. Bourgain \cite{Bourgain-LK} (see also Klartag--Milman \cite{KlartagEMilmanLowerBoundsOnZp} and Vritsiou \cite{Vritsiou-ExtendingKM} for subsequent refinements).  Throughout this work, all constants $c, C, C', \ldots$ denote positive dimension-independent numeric constants, whose value may change from one occurrence to the next. We write $A \simeq B$ to denote that $c \leq A/B \leq C$ for some numeric constants $c,C >0$. 

The $L_p$-centroid bodies of a given convex body $K$ were introduced by E. Lutwak and G. Zhang in
\cite{LutwakZhang-IntroduceLqCentroidBodies} (under different normalization). More generally, given a probability measure $\mu$ (having full-dimensional support) and $p \geq 1$, define: 
\[
h_{Z_p(\mu)}(\theta) = \left( \int_{\Real^n} \abs{\scalar{x,\theta}}^p d\mu(x) \right)^{\frac{1}{p}} \quad , \quad \theta \in \Real^n ~ .
\]
The function $\theta \mapsto h_{Z_p(\mu)}(\theta)$ is a norm on $\Real^n$, and is thus the
supporting functional of an origin-symmetric convex body $Z_p(\mu) \subseteq \Real^n$ called the $L_p$-centroid body associated to $\mu$. Note that $\mu$ is isotropic iff its barycenter is at the origin and $Z_2(\mu) = B_2^n$, the Euclidean unit-ball. For a log-concave probability measure $\mu$, we also have:
\begin{equation} \label{eq:Borell}
1 \leq p \leq q \;\;\; \Rightarrow \;\;\; Z_p(\mu) \subset Z_q(\mu) \subset C \frac{q}{p} Z_p(\mu) ~;
\end{equation}
the first inclusion follows immediately from Jensen's inequality, and the second is essentially due to Berwald \cite{BerwaldMomentComparison} and may be deduced as a consequence of Borell's lemma \cite{Borell-logconcave}, see e.g. \cite{Milman-Pajor-LK,Paouris-IsotropicTail}.

The (half) mean-width $M^*(K)$ of a convex body $K$ containing the origin is defined as:
\[
M^*(K) := \int_{S^{n-1}} h_K(\theta) d\lambda_{S^{n-1}}(\theta) ~,
\]
where $h_K(\theta) = \sup \set{\scalar{\theta,x} \; ; \; x \in K}$ is the supporting functional of $K$, $S^{n-1}$ denotes the unit Euclidean sphere and $\lambda_{S^{n-1}}$ denotes the Haar probability measure on $S^{n-1}$. 
When $K$ is in addition assumed origin-symmetric, we denote by $\norm{\cdot}_K$ the norm on $\Real^n$ whose unit-ball is $K$, and the associated normed space $(\Real^n,\norm{\cdot}_K)$ is denoted $X_K$. It was shown by T. Figiel and N. Tomczak--Jaegermann \cite{l-position} that in this case, there exists a Euclidean structure on $\Real^n$ so that $M^*(K) M^*(K^\circ) \leq C Rad(X_K)$, where $K^\circ$ is the polar body to $K$, i.e. the unit-ball of the dual norm $\norm{\cdot}_K^* = h_K$, and $Rad(X)$ denotes the norm of the Rademacher projection on $L^2(X)$ (see \cite{Pisier-Book} for more details). Equivalently, we may fix the Euclidean structure and consider linear images (``positions") of $K$. A remarkable estimate of G. Pisier \cite{Pisier-Type-Implies-K-Convex,Pisier-Book} asserts that $Rad(X) \leq C \log(n)$ for all $n$-dimensional normed spaces, thereby implying the existence of a position of $K$ so that $M^*(K) M^*(K^\circ) \leq C \log(n)$. In particular, since $M^*(K) \geq \volrad(K)$ and $M^*(K^\circ) \geq 1/\volrad(K)$ by the Urysohn and Jensen inequalities, respectively \cite{GiannopoulosMilmanHandbook}, it follows that in the minimal mean-width position of $K$ having unit volume, one has:
\[
M^*(K) \leq C \sqrt{n} Rad(X_K) \leq C' \sqrt{n} \log(n) ~.
\]
Here we denote $\volrad(A) = (\vol(A) / \vol(B_2^m))^{1/m}$, the volume-radius of a Borel set $A \subset \Real^n$ having $m$-dimensional linear hull $E$, with $\vol$ denoting the induced $m$-dimensional Lebesgue measure on $E$. An elementary computation verifies that $\vol(B_2^m)^{1/m} \simeq 1 /\sqrt{m} $. 

\subsection{Mean Width In Isotropic Position}

It is nevertheless interesting to check whether other known positions enjoy the same upper-bound on their mean-widths (see e.g. \cite{KlartagEMilman-2-convex,GiannopoulosPajorPaourisPsi2,GPV-ImprovedPsi2} for applications). Our first result asserts that up to the isotropic constant and a logarithmic factor in the dimension, this is indeed the case in the isotropic position:

\begin{thm} \label{thm:main1}
Let $K$ denote an origin-symmetric isotropic convex body in $\Real^n$. Then:
\[
M^*(K) \leq C \sqrt{n} Rad(X_K) \log(1+n) L_K \leq C' \sqrt{n} \log(1+n)^2 L_K ~.
\]
\end{thm}
Up to the $Rad(X_K) \log(1+n) L_K$ term, this bound is best possible, since by Urysohn's inequality $M^*(K) \geq \volrad(K) \simeq \sqrt{n}$. The optimality of the $L_K$ term in this bound is actually intimately connected to the Slicing Problem: removing it would imply a vast improvement over Klartag's best-known bound on the isotropic constant, namely:
\begin{eqnarray*}
\forall n\geq 1 \;\; \forall \text{ isotropic convex } K \subset \Real^n  \;\;\; M^*(K) \leq C \sqrt{n} Rad(X_K) \log(1+n) \;\;\; \Rightarrow \\
\forall n\geq 1 \;\; \forall \text{ convex } K \subset \Real^n \;\;\; L_K \leq \inf_{\lambda \in (0,1]} C^{1/\lambda} (Rad(X_K) \log(1+n))^{1+\lambda} ~;
\end{eqnarray*}
see Proposition \ref{prop:Lk} and the subsequent remark. Note that always $L_K \geq L_{B_2^n} \geq c > 0$ \cite{Milman-Pajor-LK}. As for the $Rad(X_K)$ term, its presence is natural and expected 
just as in the minimal mean-width position, as easily witnessed by testing $K = \tilde{B_1^n}$, the unit-volume homothetic copy of the unit-ball of $\ell_1^n$; indeed, $M^*(\tilde{B_1^n}) \simeq \sqrt{n} \sqrt{\log(1+n)}$, $L_{\tilde{B_1^n}} \simeq 1$ and $Rad(\ell_1^n) \simeq \sqrt{\log (1+n)}$ \cite{Milman-Schechtman-Book}. So some logarithmic dependence in $n$ must ultimately be present and cannot be completely disposed of. The additional $\log(1+n)$ term is probably non-optimal. 

The previous best-known upper-bound on the mean-width of an isotropic convex body was $M^*(K) \leq C n^{3/4} L_K$. This was first shown by M. Hartzoulaki in her Ph.D. Thesis \cite{Hartzoulaki-PhD}, by establishing that the isotropic position is a (one-sided) $1$-regular M-position (up to a factor of $L_K$), and employing Dudley's entropy estimate as in \cite{GiannopoulosMilmanMeanWidth} (see below). Other subsequent proofs include that by P. Pivovarov, who employed an approach involving random polytopes \cite{Pivovarov-IsotropicMeanWidth}. As noticed in \cite{GreekBook}, this bound is also an immediate consequence of the following sharp (up to constants) estimate of G. Paouris, valid for an arbitrary isotropic log-concave probability measure $\mu$ on $\Real^n$:
\begin{equation} \label{eq:Paouris-M*}
p \in [1,\sqrt{n}] \;\;\; \Rightarrow \;\;\; M^*(Z_p(\mu)) \leq C \sqrt{p} ~;
\end{equation}
(in fact, Paouris shows this for all $p \leq q^*(\mu)$, which is equivalent to requiring that the diameter $\diam(Z_p(\mu)) \leq c \sqrt{n}$ for an appropriately small constant $c>0$, see e.g. \cite[Section 4]{KlartagEMilmanLowerBoundsOnZp}). Indeed, by (\ref{eq:Borell}) we have
$M^*(Z_n(\mu)) \leq C \frac{n}{\sqrt{n}} M^*(Z_{\sqrt{n}}(\mu)) \leq C n^{3/4}$. It remains to note that $Z_n(\lambda_K) \simeq conv(K \cup -K)$ as an easy corollary of the Brunn--Minkowski inequality (e.g. \cite{GreekBook}).
It immediately follows that for an origin-symmetric isotropic convex body $K$:
\[
 M^*(K) \simeq L_K M^*(Z_n(\lambda_{K / L_K})) \leq C n^{3/4} L_K ~.
 \]

Our next result extends Theorem \ref{thm:main1} to an estimate on $M^*(Z_p(\mu))$ for all $p \geq 1$, thereby extending the estimate (\ref{eq:Paouris-M*}) to the range $p \geq \sqrt{n}$. Inspecting again the example of the uniform measure on $\tilde{B_1^n} / L_{\tilde{B_1^n}}$ illustrates that a logarithmic term must appear in the estimate as $p$ approaches $n$ (either directly or via the norm of the Rademacher projection), and this is indeed the case:
\begin{thm} \label{thm:main2}
Let $\mu$ denote an isotropic probability measure on $\Real^n$. Then for all $p \geq 1$:
\[
M^*(Z_p(\mu)) \leq C Rad(X_{Z_p(\mu)}) \max\brac{\frac{p \log(1+p)}{\sqrt{n}} , \sqrt{p}} ~.
\]
\end{thm}
As explained above, setting $p=n$ and $\mu = \lambda_{K / L_K}$ in Theorem \ref{thm:main2} recovers Theorem \ref{thm:main1}. Up to the $Rad(X_{Z_p(\mu)})$ term, Theorem \ref{thm:main2} recovers the sharp Paouris bound (\ref{eq:Paouris-M*}) in the range $p \in [1,\sqrt{n}]$. Note that $Rad(X_{Z_p(\mu)}) \leq C \log(1 + \min(p,n))$, see Section \ref{sec:proof}. Using in addition (\ref{eq:Borell}), we summarize the currently best-known estimates:
\begin{equation} \label{eq:M*-summary}
M^*(Z_p(\mu)) \leq C \begin{cases}  \sqrt{p} & 1 \leq p \leq \sqrt{n} \\ n^{-1/4} p &  \sqrt{n} \leq p \leq \sqrt{n} \log^2(1+n) \\
\sqrt{p} \log(1+n) & \sqrt{n} \log^2(1+n) \leq p \leq n / \log^2 (1+n) \\ \frac{p}{\sqrt{n}} \log^2(1+n) & n / \log^2 (1+n) \leq p \leq n \end{cases} ~.
\end{equation}

\subsection{Mean Width In Arbitrary Position}

In fact, Theorems \ref{thm:main1} and \ref{thm:main2} are particular cases of our main result, which we now state in full generality:

\begin{thm} \label{thm:main}
Let $\mu$ denote a log-concave probability measure on $\Real^n$ with barycenter at the origin. Let $\lambda_1^2 \geq \ldots \geq \lambda_n^2 > 0$ denote the eigenvalues of $\Cov(\mu)$. Then for any $p \geq 1$:
\begin{eqnarray*}
M^*(Z_p(\mu)) & \leq & C Rad(X_{Z_p(\mu)}) \frac{1}{\sqrt{n}} \sum_{k=1}^n \max \brac{\sqrt{\frac{p}{k}},\frac{p}{k}} (\Pi_{i=1}^{k} \lambda_i)^{\frac{1}{k}} \\
& \simeq & C' Rad(X_{Z_p(\mu)}) \frac{1}{\sqrt{n}} \sum_{i=1}^n \max \brac{\sqrt{\frac{p}{i}},\frac{p}{i}} \lambda_i  ~.
\end{eqnarray*}
\end{thm}

Using $\lambda_i \equiv 1$ in the isotropic case, Theorem \ref{thm:main2} readily follows. 

\subsection{Covering Estimates}

Recall that given two convex bodies $K,L$ in $\Real^n$, the covering number $N(K,L)$ is the minimal number of translates of $L$ whose union covers $K$. It was shown by Hartzoulaki \cite{Hartzoulaki-PhD} that an isotropic convex body $K$ in $\Real^n$ is (up to the $L_K$ term) in a (one-sided) $1$-regular $M$-position (see \cite{Pisier-Book} for history and terminology),  namely:
\begin{equation} \label{eq:Hartzoulaki}
N(K, t \sqrt{n} B_2^n) \leq \exp\brac{C n \frac{L_K}{t}} \;\;\; \forall t > 0 ~.
\end{equation}
We can now improve this for $t \geq C Rad(X_K)^2 \log^2(1+Rad(X_K)) L_K$ by simply invoking Sudakov's inequality (e.g. \cite{Pisier-Book}):
\begin{equation} \label{eq:Sudakov}
N(K , t B_2^n) \leq \exp\brac{C n \frac{M^*(K)^2}{t^2}} \;\;\; \forall t > 0 ~.
\end{equation}
Indeed, coupled with the estimate on $M^*(K)$ from Theorem \ref{thm:main1}, (\ref{eq:Sudakov}) immediately implies that an origin-symmetric isotropic convex body $K$ is, up to the $Rad(X_K) \log(1+n) L_K$ term, in a (one-sided) $2$-regular $M$-position, namely:
\[
N(K, t \sqrt{n} B_2^n) \leq \exp \brac{C n \frac{Rad(X_K)^2 \log^2(1+n) L_K^2}{t^2}} \;\;\; \forall t > 0 ~.
\]
In fact, one can actually slightly refine this covering estimate as follows:
\begin{thm} \label{thm:covering}
For all $t \in [Rad(X_K) L_K, C \sqrt{n} L_K]$ we have:
\[
N(K , t \sqrt{n} B_2^n) \leq \exp\brac{C n  \frac{Rad(X_K)^2 L_K^2}{t^2} \log^2 \brac{1+\frac{t^2}{Rad(X_K)^2 L_K^2}}} ~.
\]
\end{thm}
Similar estimates are obtained for $L_p$-centroid bodies in Section \ref{sec:proof}. 

\subsection{Main Ingredient of Proof}

We denote by $G_{n,k}$ the Grassmann manifold of all $k$-dimensional linear subspaces of $\Real^n$ ($1 \leq k \leq n$), and given $F \in G_{n,k}$, we denote by $P_F$ the orthogonal projection onto $F$. Our main result is a rather elementary consequence of the following remarkable theorem of V. Milman and G. Pisier \cite{MilmanPisier-DudleyConjecture}, as exposed in \cite[Chapter 9]{Pisier-Book}, which does not seem to be as well-known as it rightfully should:
\begin{thm}[Milman--Pisier] \label{thm:MilmanPisier}
\begin{equation} \label{eq:MilmanPisier}
\sqrt{n} M^*(K) \leq C \sum_{k=1}^n \frac{1}{\sqrt{k}} Rad_k(K) v_k(K) ~,
\end{equation}
where:
\[
v_k(K) := \sup\set{ \volrad(P_F K) ; F \in G_{n,k} } ~,
\]
and:
\[
Rad_k(K) := \sup \set{Rad(X_{P_F K}) ; F \in G_{n,k}} ~.
\]
\end{thm} 

Theorem \ref{thm:MilmanPisier} was used in \cite{MilmanPisier-DudleyConjecture} to resolve in the positive a conjecture of R. M. Dudley (see \cite{Pisier-Book}). 
Indeed, let us compare the estimate (\ref{eq:MilmanPisier}) to Dudley's entropy estimate:
\begin{equation} \label{eq:Dudley}
\sqrt{n} M^*(K) \leq C \sum_{k=1}^{\infty} \frac{1}{\sqrt{k}} e_k(K) ~,
\end{equation}
where $e_k(K) := \min\set{t > 0 \;;\; N(K, t B_2^n) \leq 2^k}$ is the $k$-th entropy number. By an elementary volumetric estimate, for all $k = 1,\dots,n$:
\[
\frac{\vol(P_F K)}{e_k(K)^k \vol(P_F B_2^n)} \leq N(P_F K, e_k(K) P_F B_2^n) \leq N(K,e_k(K) B_2^n) \leq 2^k ~,~ \forall F \in G_{n,k} ~,
\]
and therefore $v_k(K) \leq 2 e_k(K)$. Consequently, up to the $Rad_k(K)$ terms, (\ref{eq:MilmanPisier}) should be seen as a (very useful) refinement of (\ref{eq:Dudley}). 

\medskip
\textbf{Acknowledgement.} I thank Apostolos Giannopoulos and Bo'az Klartag for their comments and interest.

\section{Preliminaries}

Given $F \in G_{n,k}$, we denote by $\pi_F \mu := \mu \circ P_F^{-1}$ the push-forward of a Borel measure $\mu$ on $\Real^n$ via $P_F$. A consequence of the Prekop\'a--Leindler celebrated extension of the Brunn--Minkowski inequality (e.g. \cite{GardnerSurveyInBAMS}), is that the marginal $\pi_F \mu$ of a log-concave measure $\mu$ is itself log-concave on $F$. This is particularly useful since $P_F Z_p(\mu) = Z_p(\pi_F \mu)$, as follows directly from the definitions. 

Recall that the Banach--Mazur distance between two origin--symmetric convex bodies $K,L$ in $\Real^n$ is defined as:
\[
d_{BM}(K,L) := \inf \set{ a b \; ; \; \frac{1}{b} K \subset T(L) \subset a K ~,~ T \in GL(n)} ~.
\]
By John's Theorem (e.g. \cite{GiannopoulosMilmanHandbook}), $d_{BM}(K,B_2^n) \leq \sqrt{n}$ for any origin-symmetric convex $K \subset \Real^n$. 

As for the definition of the Rademacher projection, we refer to \cite{Pisier-Book,Milman-Schechtman-Book}. We will only require the following estimate on its norm, due to Pisier (see \cite{Pisier-Book}):
\begin{equation} \label{eq:PisierRad}
Rad(X_K) \leq C \log(1+d_{BM}(K,B_2^n)) \leq C \log(1+n) ~,
\end{equation}
where the second inequality follows by John's Theorem. In addition, it is easy to show that this norm is self-dual $Rad(X_K) = Rad(X_{K^\circ})$, and since it cannot increase by passing to a subspace, the same holds by duality when passing to a quotient space: $Rad(X_{K \cap E}) , Rad(X_{P_E K}) \leq Rad(X_K)$.  

The isotropic constant of a log-concave probability measure $\mu$ on $\Real^n$ having density $f_\mu$ is defined as the following affine-invariant quantity:
\begin{equation} \label{eq:Lmu}
L_\mu := \norm{f_\mu}_{L^\infty}^\frac{1}{n} (\det \; \Cov(\mu))^{\frac{1}{2n}} ~. 
\end{equation}
Observe that $L_{\lambda_K}$ indeed coincides with $L_K$ for a convex body $K$ in $\Real^n$. 
It was shown by K. Ball \cite{Ball-PhD,Ball-kdim-sections} that given $n \geq 1$:
\[
\sup_\mu {L_\mu} \leq C \sup_{K} L_K ~,
\]
where the suprema are taken over all log-concave probability
measures $\mu$ and convex bodies $K$ in $\Real^n$, respectively (see
e.g. \cite{KlartagPerturbationsWithBoundedLK} for the non-even case).

The fundamental estimate which we employ throughout this work, and which plays an equally fundamental role in previous groundbreaking works of G. Paouris \cite{Paouris-IsotropicTail,PaourisSmallBall} and B. Klartag \cite{Klartag-Psi2} (see also Klartag--Milman \cite{KlartagEMilmanLowerBoundsOnZp}), is given by:
\begin{thm}[Paouris, Klartag] \label{thm:Zn}
Let $\mu$ denote a log-concave probability measure on $\Real^n$ with barycenter at the origin. Then:
\[
\volrad(Z_n(\mu)) \simeq \sqrt{n} \; \frac{\det \; \Cov(\mu)^{\frac{1}{2n}}}{L_\mu} \leq C \sqrt{n} \;  \det \; \Cov(\mu)^{\frac{1}{2n}} ~.
\]
\end{thm}
\begin{proof}
For the first equivalence, see  \cite[Proposition 3.7]{PaourisSmallBall} or \cite[Lemma 2.8]{Klartag-Psi2} in the case that $\mu$ is even; in the general case, see \cite[Lemma 2.2]{KlartagPerturbationsWithBoundedLK} and the subsequent computation. The second inequality follows since $L_{\mu} \geq c > 0$ for any probability measure $\mu$, see e.g. \cite{Klartag-Psi2}. 
\end{proof}

The following corollary is due to Paouris:
\begin{cor}[Paouris] \label{cor:Zp}
With the same conditions as above, for any $p \in [1,n]$:
\[
\volrad(Z_p(\mu)) \leq C \sqrt{p} \; \det \; \Cov(\mu)^{\frac{1}{2n}} ~.
\]
\end{cor}
\begin{proof}
As both sides are invariant under volume-preserving linear transformations of $\Real^n$ and scale linearly under dilation, we may assume that $\mu$ is isotropic. The claim is then the content of \cite[Theorem 6.2]{Paouris-IsotropicTail}. Indeed, we may assume by (\ref{eq:Borell}) that $p$ is an integer, and so by Alexandrov's inequality between quermassintegrals and Kubota's formula (e.g. \cite{GiannopoulosMilmanHandbook}), we have:
\[
\volrad(Z_p(\mu)) \leq \brac{\int_{G_{n,p}} \volrad(P_F Z_p(\mu))^p d\lambda_{G_{n,p}}(F)}^{1/p} ~,
\]
where $\lambda_{G_{n,p}}$ denotes the Haar probability measure on $G_{n,p}$. Employing Theorem \ref{thm:Zn} for the isotropic log-concave measure $\pi_F \mu$ on $F \in G_{n,p}$, we see that $\volrad(P_F Z_p(\mu)) = \volrad(Z_p(\pi_F \mu)) \leq C \sqrt{p}$, and so the conclusion follows. 
\end{proof}

\section{Proofs} \label{sec:proof}

Our computations are based on the following immediate corollary of Theorem \ref{thm:Zn} and Corollary \ref{cor:Zp}:
\begin{prop} \label{prop:main}
Let $\mu$ denote a log-concave probability measure on $\Real^n$ with barycenter at the origin. Let $p \geq 1$ and $k = 1,\ldots,n$. Then:
\[
v_k(Z_p(\mu))  \leq C \sqrt{\frac{p}{k}} \max(\sqrt{p},\sqrt{k}) \max_{F \in G_{n,k}} \det \; \Cov(\pi_F \mu)^{\frac{1}{2k}} ~.
\]
\end{prop}
\begin{proof}
Let $F \in G_{n,k}$. When $k \leq p$ we use (\ref{eq:Borell}) and Theorem \ref{thm:Zn}:
\[
\volrad(P_F(Z_p(\mu)))\leq C \frac{p}{k} \volrad(P_F(Z_k(\mu))) = C \frac{p}{k} \volrad(Z_k(\pi_F \mu)) \leq C' \frac{p}{k} \sqrt{k} \; \det \; \Cov(\pi_F \mu)^{\frac{1}{2k}} ~.
\]
When $k \geq p$ we use Corollary \ref{cor:Zp}:
\[
\volrad(P_F(Z_p(\mu))) = \volrad(Z_p(\pi_F \mu)) \leq C \sqrt{p} \; \det \; \Cov(\pi_F \mu)^{\frac{1}{2k}} ~.
\]
Combining the two cases and maximizing over $F \in G_{n,k}$, the assertion immediately follows. 
\end{proof}
\begin{rem}
When $\mu$ is the uniform measure on an isotropic convex body, this estimate was already deduced in \cite[Theorem 2.4]{GiannopoulosPajorPaourisPsi2} using a slightly different argument.
\end{rem}

\subsection{Proof of Theorem \ref{thm:main}}

By the Milman--Pisier Theorem \ref{thm:MilmanPisier}:
\begin{equation} \label{eq:MP}
\sqrt{n} M^*(Z_p(\mu)) \leq C \sum_{k=1}^n \frac{1}{\sqrt{k}} Rad_k(Z_p(\mu)) v_k(Z_p(\mu)) ~.
\end{equation}
Obviously $Rad_k(Z_p(\mu)) \leq Rad(X_{Z_p(\mu)})$ by passing to a quotient space. 
\begin{comment}
(the norm of the Rademacher projection is self-dual, and passing to a subspace cannot increase this norm). Recall that Pisier's estimate on the norm of the Rademacher projection is in fact (see \cite{PisierBook}):
\[
Rad(X_K) \leq C \log(1+d_{BM}(K,B_2^n)) ~,
\]
where $d_{BM}(K,L) := \inf \set{ a b \; ; \; \frac{1}{b} K \subset T(L) \subset a K ~,~ T \in GL(n)}$ denotes the Banach--Mazur distance between two origin-symmetric convex bodies $K,L$ in $\Real^n$. 
By John's Theorem \cite{XXXX} $d_{BM}(P_F Z_p(\mu),B_2(F)) \leq \sqrt{k}$, and therefore $Rad_k(Z_p(\mu)) \leq C \log(1+k)$. 
\end{comment}
Note that by (\ref{eq:Borell}) we know that:
\[
\frac{1}{C} Z_2(\mu) \subset Z_p(\mu) \subset C p Z_2(\mu) ~,
\]
for $p \geq 1$, and since $Z_2(\mu)$ is an ellipsoid (the Legendre ellipsoid of inertia), it follows by Pisier's estimate (\ref{eq:PisierRad}) that:
\[
 Rad(X_{Z_p(\mu)}) \leq C' \log(1+d_{BM}(Z_p(\mu) ,Z_2(\mu))) \leq C' \log(1+ p) ~.
\]
On the other hand, $Rad_k(L) \leq C \log(1+k)$ for any origin-symmetric convex $L$ by applying Pisier's estimate coupled with John's Theorem (\ref{eq:PisierRad}) in dimension $k$. 
Consequently:
\[
Rad_k(Z_p(\mu)) \leq \min(Rad(X_{Z_p(\mu)}), C \log(1+k)) \leq C' \log(1 + \min(k,p)) ~. 
\]
However, as one may check, there will be no loss in the final estimate in using the trivial $Rad_k(Z_p(\mu)) \leq Rad(X_{Z_p(\mu)})$. 

\medskip

Now, recalling that $\lambda_1^2 \geq \lambda_2^2 \geq \ldots \geq \lambda_n^2 > 0$ denote the eigenvalues of $\Cov(\mu)$, we have by Proposition \ref{prop:main}:
\[
v_k(Z_p(\mu))  \leq C \sqrt{\frac{p}{k}} \max(\sqrt{p},\sqrt{k}) \max_{F \in G_{n,k}} \det \; \Cov(\pi_F \mu)^{\frac{1}{2k}}  \leq C \sqrt{\frac{p}{k}} \max(\sqrt{p},\sqrt{k})(\Pi_{i=1}^{k} \lambda_i)^{1/k} ~;
\]
(the last inequality is an elementary exercise in linear algebra, for which it may be useful to recall the Cauchy--Binet formula). 
The first inequality asserted in Theorem \ref{thm:main} then immediately follows from (\ref{eq:MP}):
\begin{eqnarray*}
\sqrt{n} M^*(Z_p(\mu)) &\leq&  C Rad(Z_p(\mu)) \sum_{k=1}^n \max\brac{\sqrt{\frac{p}{k}},\frac{p}{k}}(\Pi_{i=1}^{k} \lambda_i)^{1/k} ~;
\end{eqnarray*}
to get a slightly more aesthetically pleasing bound, we may apply the Arithmetic-Geometric Means Inequality and proceed to estimate:
\begin{eqnarray*}
&\leq& C Rad(Z_p(\mu)) \sum_{k=1}^n \max\brac{\sqrt{\frac{p}{k}},\frac{p}{k}} \frac{1}{k} \sum_{i=1}^k \lambda_i \\
& = & C Rad(Z_p(\mu)) \sum_{i=1}^n \lambda_i \sum_{k=i}^n \max\brac{\sqrt{\frac{p}{k}},\frac{p}{k}} \frac{1}{k} \\
& \leq & C' Rad(Z_p(\mu))  \sum_{i=1}^n \max\brac{\sqrt{\frac{p}{i}},\frac{p}{i}} \lambda_i \\
& \leq & C' Rad(Z_p(\mu))  \sum_{i=1}^n \max\brac{\sqrt{\frac{p}{i}},\frac{p}{i}} (\Pi_{j=1}^{i} \lambda_j)^{1/i} ~, 
\end{eqnarray*}
and so the equivalent bound using the arithmetic average follows. The proof of Theorem \ref{thm:main} is complete.

\subsection{Proof of Theorem \ref{thm:covering}}

To show the statement about regularity, we use the following corollary of a slightly more general version of the Milman--Pisier Theorem (see Section \ref{sec:conclude}) coupled with the Pajor--Tomczak-Jaegermann refinement of V. Milman's low-$M^*$-estimate, which reads as follows (see the proof of \cite[Corollary 9.7]{Pisier-Book}):
\[
k^{1/2} c_{k}(K) \leq  C \sum_{j=\lfloor c k \rfloor}^n \frac{1}{\sqrt{j}} Rad_j(K) v_j(K) \;\;\; \forall k = 1,\ldots,n-1 ~,
\]
where:
\[
c_k(K) := \inf \set{ \diam(K \cap F) \; ; \; F \in G_{n,n-k}} ~.
\]
Consequently, by Carl's Theorem \cite[Theorem 5.2]{Pisier-Book} and Proposition \ref{prop:main} applied to $\mu = \lambda_{K}$ and $p=n$, we obtain:
\begin{eqnarray*}
& & \sup_{k = 1,\ldots,n} \frac{k^{1/2}}{\log(1+n/k)} e_k(K) \leq  C' \sup_{k = 1,\ldots,n-1} \frac{k^{1/2}}{\log(1+n/k)} c_k(K) \\
&\leq &  C'' \sup_{k = 1,\ldots,n-1} \frac{1}{\log(1+n/k)} \sum_{j=\lfloor c k \rfloor}^n \frac{1}{\sqrt{j}} Rad_j(K) v_j(K) \\
& \leq & C''' \sup_{k = 1,\ldots,n-1} \frac{n Rad(X_K) L_K}{ \log(1+n/k)} \sum_{j=\lfloor c k \rfloor}^n \frac{1}{j} \leq C'''' n Rad(X_K) L_K ~.
\end{eqnarray*}
In other words:
\[
N(K , C \sqrt{n} Rad(X_K) L_K  \sqrt{n/k} \log(1+n/k) B_2^n) \leq 2^k \;\;\; \forall k = 1,\ldots,n ~.
\]
Setting $t = Rad(X_K) L_K \sqrt{n/k} \log(1+n/k)$, Theorem \ref{thm:covering} immediately follows.
Note that in isotropic position $K \subset C n L_K B_2^n$ \cite{Milman-Pajor-LK,KLS}, and so $N(K, t \sqrt{n} B_2^n) = 1$ for $t \geq C \sqrt{n} L_K$.

\subsection{Covering $L_p$-centroid bodies}

Similar covering estimates may be deduced for $Z_p(\mu)$. The previous best-known estimate for these covering estimates is due to Giannopoulos--Paouris--Valettas \cite{GPV-ImprovedPsi2}, who showed that for any isotropic log-concave measure $\mu$ on $\Real^n$ and $p \in [1,n]$:
\begin{equation} \label{eq:GPV}
N(Z_p(\mu),C_1 t \sqrt{p} B_2^n) \leq \exp\brac{C_2 \frac{n}{t^2} + C_3 \frac{\sqrt{n} \sqrt{p}}{t}}  \;\;\; \forall t \geq 1 ~.
\end{equation}
Note that since $Z_p(\mu) \subset C p Z_2(\mu) = C p B_2^n$ by (\ref{eq:Borell}), it is enough to only test $t \in [1,\sqrt{p}]$. 
Also note that setting $\mu = \lambda_K$ and $p=n$, this recovers Hartzoulaki's estimate (\ref{eq:Hartzoulaki}). 

Invoking Sudakov's inequality (\ref{eq:Sudakov}) and using the estimate (\ref{eq:M*-summary}) on $M^*(Z_p(\mu))$, an improved covering estimate immediately follows when $t \geq \sqrt{n/p} \log^2(1+n)$. Summarizing the resulting presently best-known estimates when $p \in [1, n /\log^2(1+n)]$ and $t \in [1,\sqrt{p}]$, we have:
\[
\log N(Z_p(\mu),C_1 t \sqrt{p} B_2^n) \leq 
\begin{cases} 
C_2 \frac{n}{t^2} &  1 \leq t \leq \sqrt{n/p} \\    
C_3 \frac{\sqrt{n} \sqrt{p}}{t} &  \sqrt{n/p}  \leq t \leq \sqrt{n/p} \log^2 (1+n) \\
C_4 \frac{n \log^2(1+n)}{t^2} &  \sqrt{n/p} \log^2(1+n) \leq t \leq \sqrt{p} 
\end{cases} ~.
\]
When $n / \log^2(1+n) \leq p \leq n$ one may obtain a slight further improvement beyond Sudakov's inequality, by invoking an argument similar to the one used in the proof of Theorem \ref{thm:covering}; we leave this to the interested reader.

\section{Concluding Remarks} \label{sec:conclude}

\subsection{Extended Milman--Pisier Theorem}
For completeness, we mention that the Milman--Pisier Theorem \ref{thm:MilmanPisier} is in fact a bit more general (see \cite[Theorem 9.1]{Pisier-Book}):
\begin{thm}[Milman--Pisier]
For any origin-symmetric convex body in $\Real^n$ and $j = 0,\ldots,n-1$:
\[
\sqrt{n-j} M_{n-j}^*(K) \leq C \sum_{k= \lfloor c j + 1 \rfloor }^n \frac{1}{\sqrt{k}} Rad_k(K) v_k(K) ~,
\]
where:
\[
M_{m}^*(K) := \inf \set{ M^*(P_F K) \; ; \; F \in G_{n,m}} ~.
\]
\end{thm}
Here $M^*(L)$ denotes the (half) mean-width of $L$ in its linear hull $F \in G_{n,m}$, namely:
\[
M^*(L) = \int_{S^{n-1} \cap F} h_L(\theta) d\sigma_{S^{n-1} \cap F}(\theta) ~,
\]
where $\sigma_{S^{n-1} \cap F}$ denotes the corresponding Haar probability measure. Plugging in the estimates of the previous section, one immediately obtains upper bounds on $M_{n-j}^*(Z_p(\mu))$; we leave this again to the interested reader. 

\subsection{Removing non-optimal terms}

Most probably the $\log(1+n)$ term which appears in our estimates is non-optimal. This is in contrast with the norm of the Rademacher projection term, which should play a role in the estimates (although perhaps with a different power), as in the best-known estimate for the minimal mean-width. To remove the $\log(1+n)$ term, perhaps a majoring-measures type approach in the spirit of Talagrand (see \cite{Talagrand-Book}) would be successful. However, this seems difficult at this point. 

As for the $L_K$ term, whether it is possible to remove it from our estimate on the mean-width is intimately connected to the Slicing Problem. We refer the reader to the PhD Thesis of K. Ball \cite{Ball-PhD}, who showed that when the isotropic constant is bounded then the isotropic position is an $M$-position, and to the work of Bourgain, Klartag and Milman \cite{BKM-symmetrizations}, who conversely showed that if the isotropic position is always an $M$-position, then the isotropic constant is universally bounded. For completeness, we recall the corresponding arguments: 

\begin{prop} \label{prop:Lk}
Denote:
\begin{eqnarray*}
e_{m}^\wedge & := & \sup \set{ e_m(K) / \sqrt{m} \; ; \; K \text{ is an isotropic convex body in $\Real^m$} } ~, \\
L_n & := & \sup \set{ L_K \; ; \; K \text{ is an (isotropic) convex body in $\Real^n$}} ~.
\end{eqnarray*}
Then:
\[
L_n \leq \inf_{\lambda \in (0,1]} C^{1/\lambda} (e_{\lfloor (1+\lambda) n \rfloor}^\wedge)^{1+\lambda} ~.
\]
Conversely, for any isotropic convex body $K$ of volume one in $\Real^n$:
\[
e_n(K) \leq C L_K \sqrt{n} ~,
\]
and hence:
\[
e_n^\wedge \leq C L_n ~.
\]
\end{prop}
\begin{proof}
The second assertion follows from the work of Ball \cite{Ball-PhD}, who showed that when $L_K$ is bounded, the isotropic position is an $M$-position. Further refinements pertaining to regularity were obtained by Hartzoulaki  (\ref{eq:Hartzoulaki}), Giannopoulos--Paouris--Pajor \cite{GiannopoulosPajorPaourisPsi2} and Giannopoulos--Paouris--Valettas (\ref{eq:GPV}). Any of these results implies in particular that $e_n(K) \leq C \sqrt{n} L_K$. 

To show the first assertion, we use a small variation on the argument from \cite{BKM-symmetrizations}. Let $K$ denote an isotropic convex body in $\Real^n$. Given $m \geq n$, denote by $Q_m$ the following convex body in $\Real^m$:
\[
Q_{m} := \brac{\frac{L_{D_{m-n}}}{L_K}}^{\frac{m-n}{m}} K \times \brac{\frac{L_K}{L_{D_{m-n}}}}^{\frac{n}{m}} D_{m-n} ~,
\]
where $D_{m-n}$ is the homothetic copy of $B_2^{m-n}$ having volume one. It is immediate to verify that $Q$ is isotropic, and consequently $e_m(Q_m) \leq e_m^\wedge \sqrt{m}$.  Denoting by $E$ the subspace spanned by the last $m-n$ coordinates and by $B_E$ its unit Euclidean ball, it is straightforward to verify:
\[
N(Q_m \cap E, e_m(Q_m) B_E) \leq N(Q_m,e_m(Q_m) B_2^m) \leq 2^m ~.
\]
On the other hand, a trivial volumetric estimate yields:
\[
N(Q_m \cap E, e_m(Q_m) B_E) ^{\frac{1}{m-n}} \geq \frac{\volrad \brac{\brac{\frac{L_K}{L_{D_{m-n}}}}^{\frac{n}{m}} D_{m-n} }}{\volrad(e_m(Q_m) B_2^{m-n})} \geq \frac{\brac{\frac{L_K}{L_{D_{m-n}}}}^{\frac{n}{m}}  c \sqrt{m-n}}{\sqrt{m} e_m^\wedge} ~.
\]
Combining both estimates and denoting $\lambda = \frac{m-n}{n}$, it follows that:
\[
L_K \leq L_{D_{\lambda n}} \brac{\frac{1}{c} 2^{\frac{1+\lambda}{\lambda}}  \sqrt{\frac{1+\lambda}{\lambda}} } ^{1+\lambda} \brac{e_{n(1+\lambda)}^\wedge}^{1+\lambda} ~.
\]
Since $L_{D_m} \simeq 1$ uniformly in $m$, the first assertion follows. 
\end{proof}

\begin{rem}
Since $e_n(K) \leq C M^*(K)$ by Sudakov's inequality (\ref{eq:Sudakov}), it follows that if we could remove the $L_K$ term from our upper bound on $M^*(K)$ given in Theorem \ref{thm:main1}, namely, if:
\[
 M^*(K) \leq C \sqrt{n} Rad(X_K) \log(1+n) \leq C' \sqrt{n} \log(1+n)^2 ~,
 \]
for any $n \geq 1$ and origin-symmetric isotropic convex $K$ in $\Real^n$, then we would obtain $L_n \leq \log(1+n)^2 C^{\sqrt{\log (\log(e+n)^2)}}$. In fact, inspecting the proof, we would obtain: 
\[
L_K\leq Rad(X_K) \log(1+n) C^{\sqrt{\log (Rad(X_K)\log(e+n))}} ~,
\]
since it is easy to verify that $Rad(X_{Q_m}) \simeq Rad(X_K)$, uniformly in $m$. 
\begin{comment}
Furthermore, we know by Theorem \ref{thm:covering} that $e_n(K) \leq C \sqrt{n} Rad(X_K) L_K$ in isotropic position. If we could remove the $L_K$ term, namely, if:
\[
e_n(K) \leq C \sqrt{n} Rad(X_K) ~,
\]
for any $n \geq 1$ and origin-symmetric isotropic convex $K$ in $\Real^n$, then it would follow that:
\[
 L_K\leq Rad(X_K) C^{\sqrt{\log (e+Rad(X_K))}} ~.
 \]
 \end{comment}
\end{rem}

\setlength{\bibspacing}{2pt}

\vspace{-20pt}

\bibliographystyle{plain}
\bibliography{../../../ConvexBib}

 \end{document}